\documentclass{amsart}

\usepackage[active]{srcltx}

\usepackage{graphicx}

\usepackage{latexsym}
\usepackage{amsmath,amsthm}
\usepackage{amsfonts}
\usepackage[psamsfonts]{amssymb}
\usepackage{enumerate}

\usepackage{epstopdf}

\usepackage[final]
{showkeys}
\usepackage{url}

\RequirePackage[OT1]{fontenc}

\usepackage{calrsfs}

\usepackage{hyperref}

\theoremstyle{plain} 
\newtheorem{theorem}{Theorem}

\newtheorem{lemma}[theorem]{Lemma}
\newtheorem{proposition}[theorem]{Proposition}

\theoremstyle{definition} 

\theoremstyle{definition} 

\theoremstyle{remark} 

\theoremstyle{remark} 
\newtheorem{remark}[theorem]{Remark}
\newtheorem*{remark*}{Remark}

\makeatletter
  \renewcommand\section{\@startsection {section}{1}{\z@}%
                                   {-\bigskipamount}%
                                   {\medskipamount}%
                                   {\large\bfseries
                                   \raggedright}}

  \renewcommand\subsection{\@startsection {subsection}{2}{\z@}%
                                   {-\medskipamount}%
                                   {\smallskipamount}%
                                   {\bfseries
                                   \raggedright}}
\makeatother

\newcommand{\argmax}{\operatorname{argmax}}

\renewcommand{\gg}{>\kern-2pt>}
\renewcommand{\ll}{<\kern-2pt<}

\newcommand{\simpl}{\mathsf{simpl}}

\newcommand{\FM}{\mathsf{FM}}

\newcommand{\ceil}[1]{\lceil #1 \rceil}

\renewcommand{\gg}{>\kern-2pt>}
\renewcommand{\ll}{<\kern-2pt<}

\newcommand{\dd}{\partial}
\renewcommand{\dd}{\operatorname{d}\!}

\renewcommand{\le}{\leqslant}
\renewcommand{\ge}{\geqslant}

\newcommand{\al}{\alpha}
\newcommand{\be}{\beta}

\newcommand{\Si}{\Sigma}

\newcommand{\de}{\delta}
\newcommand{\De}{\Delta}

\newcommand{\Om}{\Omega}
\newcommand{\om}{\omega}

\newcommand{\LD}{\mathcal{L}\!\mathcal{D}}
\renewcommand{\LD}{\mathcal{L}{\kern -1.9pt}\mathcal{D}}
\renewcommand{\LD}{\mathcal{D}}
\renewcommand{\LD}{\mathcal{L}{\kern -4pt}\mathcal{C}}
\renewcommand{\LD}{\mathcal{R}{\kern -3pt}\mathcal{C}}

\newcommand{\ii}[1]{\mathrm{I}\!\left\{#1\right\}}
\renewcommand{\ii}[1]{\,\mathrm{I}\!\left\{#1\right\}}

\renewcommand{\P}{\operatorname{\mathsf{P}}} 

\newcommand{\E}{\operatorname{\mathsf{E}}}

\newcommand{\F}[1]{\mathcal{F}_+^{#1}}
\renewcommand{\F}{\mathcal{F}}

\renewcommand{\d}{\mathrm{d}}

\newcommand{\vp}{\varepsilon}

\begin{document}


\title{Exact upper and lower bounds on the misclassification probability}


\author{Iosif Pinelis}

\address{Department of Mathematical Sciences\\
Michigan Technological University\\
Hough\-ton, Michigan 49931, USA\\
E-mail: ipinelis@mtu.edu}

\keywords{Classification problem, 
misclassification probability, total variation norm, 
Bayes estimators, maximum likelihood estimators}

\subjclass[2010]{Primary: 60E15, 62H30, 91B06. Secondary: 26D15, 26D20, 62C10, 68T10}



\begin{abstract} 
Exact lower and upper bounds on the best possible misclassification probability for a finite number of classes are obtained in terms of the total variation norms of the differences between the sub-distributions over the classes. 
These bounds are compared with the exact bounds in terms of the conditional entropy obtained by Feder and Merhav.  
\end{abstract}

\maketitle

\tableofcontents

\section{Introduction, summary and discussion
}\label{intro} 

Let $X$ and $Y$ be random variables (r.v.'s) defined on the same probability space $(\Om,\F,\P)$, $X$ with values in 
a set $S$ (endowed with a sigma-algebra $\Si$) and $Y$ with values in the set $[k]:=\{1,\dots,k\}$, where $k$ is a natural number; to avoid trivialities, assume $k\ge2$. 

The sets $\Om$ and $[k]$ may be regarded, respectively, as the population of objects of interest and the set of all possible classification labels for those objects. For each ``object'' $\om\in\Om$, the corresponding values $X(\om)\in S$ and $Y(\om)\in[k]$ of the r.v.'s $X$ and $Y$ may be interpreted as the (correct) description of $\om$ and the (correct) classification label for $\om$, respectively. 

Alternatively, $Y(\om)$ may be interpreted as the signal entered at the input side of a device -- with its possibly corrupted, output version $X(\om)$. 

The problem is to find a good or, better, optimal way to reconstruct, for each $\om\in\Om$, the correct label (or input signal) $Y(\om)$ based on the description (or, respectively, the output signal) $X(\om)$. 
To solve this problem, one uses a 
measurable function $f\colon S\to[k]$, referred to as a classification rule or, briefly, a classifier, which assigns a label (or an input signal) $f(x)\in[k]$ to each possible description (or, respectively, to each possible output signal) $x\in S$. 
Then 
\begin{equation*}
	p_f:=\P(f(X)\ne Y)
\end{equation*}
is the misclassification probability for the classifier $f$. 

For each $y\in[k]$, let $\mu_y$ be the sub-probability measure on $\Si$ defined by the condition
\begin{equation}\label{eq:mu_i}
\mu_y(B):=\P(Y=y,X\in B)	
\end{equation}
for $B\in\Si$, so that 
\begin{equation}\label{eq:mu:=}
	\mu:=\mu_1+\dots+\mu_k
\end{equation}
is the probability measure that is the distribution of $X$ in $S$, and let 
\begin{equation*}
	\rho_y:=\frac{\d\mu_y}{\d\mu},
\end{equation*}
the density of $\mu_y$ with respect to $\mu$. 

The value $y\in[k]$ may be considered a parameter, so that the problem may be be viewed as one of Bayesian estimation (of a discrete parameter, with values in the finite set $[k]$). 
If the r.v.\ $X$ is discrete as well, then of course 
\begin{equation*}
	\rho_y(x)=\P(Y=y|X=x)
\end{equation*}
for each $x\in S$ with $\P(X=x)\ne0$. So, for each such $x$, the function $y\mapsto\rho_y(x)$ may be referred to as the probability mass function of the posterior distribution of the parameter corresponding to the observation $x$. 

The following proposition is, essentially, a well-known fact of Bayesian estimation: 

\begin{proposition}\label{prop:f_*}
For each $x\in S$, let $f_*(x):=\min\argmax_y\rho_y(x)$, where \break  $\argmax_y\rho_y(x):=\{y\in[k]\colon\rho_y(x)=\max_{z\in[k]}\rho_z(x)\}$; thus, $f_*(x)$ is the smallest maximizer of $\rho_y(x)$ in $y\in[k]$. Then the function $f_*$ is a classifier, and 
\begin{equation*}
	p_*:=p_{f_*}=1-\int_S\max\limits_{y=1}^k\rho_y(x)\,\mu(\d x)\le p_f
\end{equation*}
for any classifier $f$, so that $p_*$ is the smallest possible misclassification probability. 
\end{proposition}

The proofs of all statements that may need a proof are deferred to Section~\ref{proofs}. 

Let 
\begin{equation}\label{eq:De}
	\De:=\sum_{1\le y<z\le k}\|\mu_y-\mu_z\|=\sum_{1\le y<z\le k}|\rho_y-\rho_z|\d\mu, 
\end{equation}
where $\|\cdot\|$ is the total variation norm. 

In the ``population'' model, the measure $\mu_y$ conveys two kinds of information: (i) the relative size $\frac{\|\mu_y\|}{\|\mu\|}=\|\mu_y\|$ (of the set of all individual descriptions) of the $y$th subpopulation (of the entire population $\Om$) consisting of the objects that carry the label $y$ and (ii) the (conditional) probability distribution $\frac{\mu_y}{\|\mu_y\|}$ of the object descriptions in this $y$th subpopulation, assuming the size $\|\mu_y\|$ of the $y$th subpopulation is nonzero. Everywhere here, $y$ and $z$ are in the set $[k]$. Thus, $\De$ is a summary characteristic of the pairwise differences between the $k$ subpopulations, which takes into account both of the two just mentioned kinds of information. 

In the input-output model, the $\|\mu_y\|$'s are interpreted as the prior probabilities of the possible input signals $y\in[k]$ -- whereas, for each $y\in[k]$, the (conditional) probability distribution $\frac{\mu_y}{\|\mu_y\|}$ is the distribution of the output signal corresponding to the given input $y$. Thus, here $\De$ is a summary characteristic of the pairwise differences between the $k$ sets of possible outputs corresponding to the $k$ possible inputs. 


\begin{remark}\label{rem:De}
By \eqref{eq:De},  
\begin{equation*}
	0\le\De\le\sum_{1\le y<z\le k}(\|\mu_y\|+\|\mu_z\|)=(k-1)\sum_1^k\|\mu_y\|=k-1.
\end{equation*}
Moreover, the extreme values $0$ and $k-1$ of $\De$ are attained, respectively,  when the measures $\mu_y$ are the same for all $y\in[k]$ and when these measures are pairwise mutually singular.
\end{remark}

The main result of this paper provides the following upper and lower bounds on the smallest possible misclassification probability $p_*$ in terms of $\De$: 
\begin{theorem}\label{th:} 
One has  
\begin{equation}\label{eq:}
	L(\De)\le p_*\le U(\De)\le U_\simpl(\De),  
\end{equation}
where 
\begin{equation}\label{eq:L,U}
\begin{aligned}
	L(\De)&:=L_k(\De):=1-\frac{1+\De}k,\\  
	U(\De)&:=U_k(\De):=1-\frac{k+1+\De-2\ceil\De}{(k-\ceil\De)(k+1-\ceil\De)},\\
	U_\simpl(\De)&:=U_{k;\simpl}(\De):=1-\frac1{k-\De}, 
\end{aligned}	
\end{equation}
and $\ceil\cdot$ is the ceiling function, so that $\ceil\De$ is the smallest integer that is no less than $\De$. 
\end{theorem}

Theorem~\ref{th:} is complemented by 

\begin{proposition}\label{prop:exact}
For each possible value of $\De$ in the interval $[0,k-1]$, the lower and upper bounds $L(\De)$ and $U(\De)$ on $p_*$ are exact: For each $\De\in[0,k-1]$, there are r.v.'s $X$ and $Y$ as described in the beginning of this paper for which one has the equality $p_*=L(\De)$; similarly, with $U(\De)$ in place of $L(\De)$. More specifically, 
the first (respectively, second) inequality in \eqref{eq:} turns into the equality if and only if there is a set $S_0\in\Si$ such that $\mu(S_0)=0$ and for each $x\in S\setminus S_0$ the values $\rho_1(x),\dots,\rho_k(x)$ constitute a permutation of numbers $a_1,\dots,a_k$ as in \eqref{eq:a_i-low} (respectively, in \eqref{eq:a_i-hi}) with $d=\De$. 
The simpler/simplified upper bound $U_\simpl(\De)$ is exact only for the integral values of $\De$. 
\end{proposition}

\begin{remark}\label{rem:bounds}
In view of Remark~\ref{rem:De}, the functions $L$, $U$, and $U_\simpl$, introduced in
Theorem~\ref{th:}, are well defined on the interval $[0,k-1]$. Moreover, $U(\De)$ is the linear
interpolation of $U_\simpl(\De)$ over the possible integral values $0,\dots,k-1$ of $\De$. Thus,
each of the functions $L$, $U$, and $U_\simpl$ is concave and strictly decreasing (from $1-\frac1k$ 
to $0$) on the interval $[0; k-1]$; moreover, the function $L$ is obviously affine.
We see that, the greater is the characteristic $\De$ of the pairwise differences between the $k$ subpopulations, the smaller are the lower and upper bounds $L(\De)$, $U(\De)$, and $U_\simpl(\De)$ on the misclassification probability $p_*$. Of course, this quite corresponds to what should be expected of good bounds on $p_*$. 
It also follows that one always has 
\begin{equation}\label{eq:p_*}
	0\le p_*\le1-\frac1k,  
\end{equation}
and the extreme values $0$ and $1-\frac1k$ of the misclassification probability $p_*$ are attained when, respectively, $\De=k-1$ and $\De=0$. 
The bounds $L$, $U$, and $U_\simpl$ are illustrated in Figure~\ref{fig:bounds(De)}.
\end{remark}

\begin{figure}[h]
	\centering
		\includegraphics[width=0.60\textwidth]{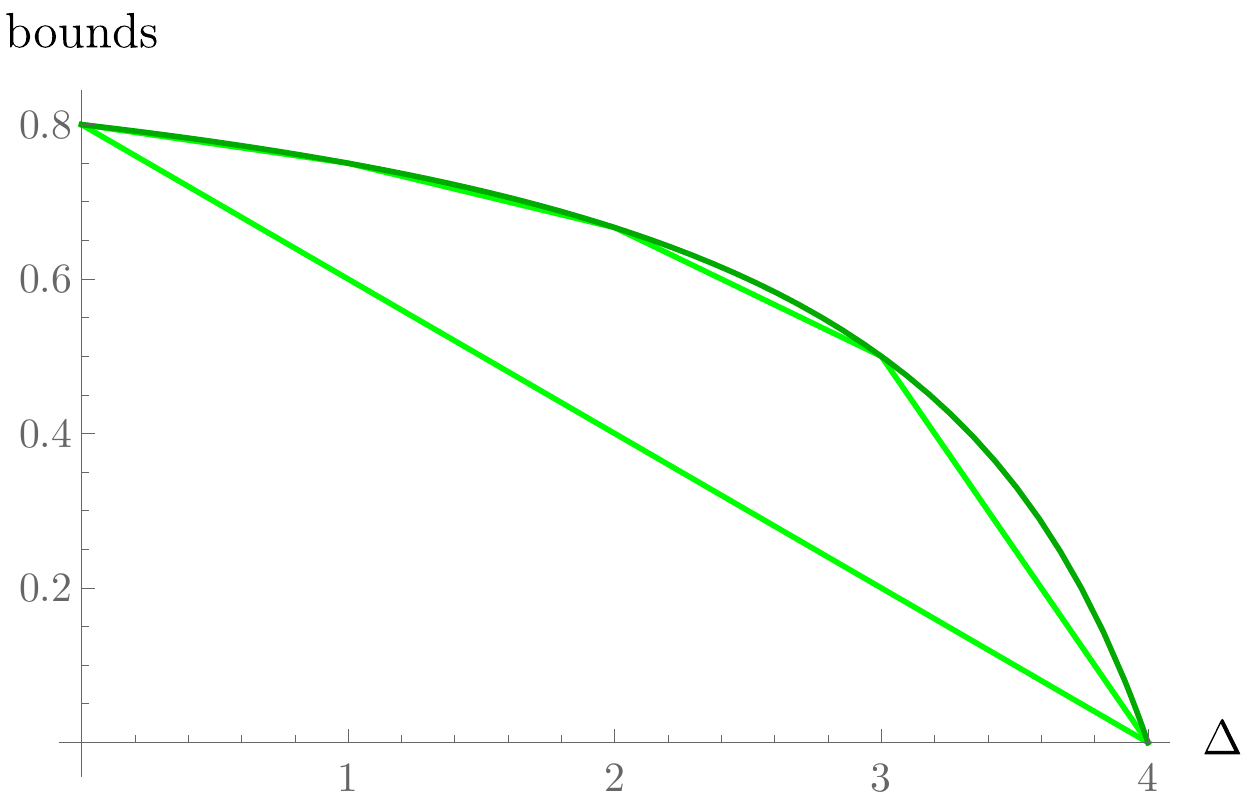}
	\caption{Graphs of the bounds $L$ (green), $U$ (green), and $U_\simpl$ (dark green) for $k=5$.}
	\label{fig:bounds(De)}
\end{figure}

Feder and Merhav \cite{fed-merhav} obtained the following exact upper and lower bounds of the optimal misclassification probability in terms of the conditional entropy $H$: 
\begin{equation*}
	L_\FM(H)\le p_*\le U_\FM(H), 
\end{equation*}
where 
\begin{equation}\label{eq:H}
	H:=H(Y|X):=-\E\sum_{y=1}^k \rho_y(X)\ln\rho_y(X)
	=-\int_S\mu(\dd x)\sum_{y=1}^k \rho_y(x)\ln\rho_y(x), 
\end{equation}
\begin{equation}\label{eq:L_FN}
	L_\FM(H):=\Phi^{-1}(H),\quad \Phi(p):=p\ln(k-1)+h_2(p),\quad h_2(p):=-p\ln p-(1-p)\ln(1-p) 
\end{equation}
for $p\in(0,1)$, $h_2(0):=0$, $h_2(1):=0$, 
\begin{equation}\label{eq:U_FN}
	U_\FM(H):=\frac{e(H)-1}{e(H)}+\frac1{e(H) (e(H)+1)} \, 
	\frac{H-\ln e(H)}{\ln (1+1/e(H))}, 
\end{equation}
and
\begin{equation}\label{eq:e(H)}
	e(H):=\ceil{e^H}-1. 
\end{equation}
Note that $\Phi(p)$ strictly and continuously increases from $0$ to $\ln k$ as $p$ increases from $0$ to $1-\frac1k$. Therefore and because all the values of the conditional entropy $H$ lie between $0$ and $\ln k$, the expression $\Phi^{-1}(H)$ is well defined, and its values lie between $0$ and $1-\frac1k$ -- which is in accordance with \eqref{eq:p_*}.  

Throughout this paper, we use only natural, base-$e$ logarithms. In \cite{fed-merhav}, the bounds are stated in terms of binary, base-$2$ logarithms. To rewrite $L_\FM(H)$ and $U_\FM(H)$ in terms of binary logarithms, replace all the instances of $\ln=\log_e$ in \eqref{eq:H}--\eqref{eq:U_FN} by $\log_2$ and, respectively, replace $e^H$ in \eqref{eq:e(H)} by $2^H$. An advantage of using natural logarithms is that then the expressions for the corresponding derivatives, used in our proofs, are a bit simpler; also, $\ln$ is a bit shorter in writing than $\log_2$ or even $\log$. 

Note also that, in the notation in \cite{fed-merhav}, the roles of $X$ and $Y$ are reversed: there, $X$ denotes the input and $Y$ the output. Our notation in this paper is in accordance with the standard convention in machine learning; cf.\ e.g.\ \cite{massart-nedelec,pin-kontor}. 

Let us compare, in detail, our ``$\De$-bounds'' $L(\De)$, $U(\De)$, and $U_\simpl(\De)$ with the ``$H$-bounds'' $L_\FM(H)$ and $U_\FM(H)$. We shall be making the comparisons only in the ``pure'' settings, when the set $\{\rho_1(x),\dots,\rho_k(x)\}$ is the same for all $x\in S$, that is, when for each $x\in S$ the $k$-tuple $(\rho_1(x),\dots,\rho_k(x))$ is a permutation of one and the same $k$-tuple $(a_1,\dots,a_k)$ (of nonnegative real numbers $a_1,\dots,a_k$ such that $a_1+\dots+a_k=1$). A reason for doing so is that one may expect the comparisons to be of greater contrast in the ``pure'' settings than in ``mixed'', non-``pure'' ones. Thus, focusing on ``pure'' settings will likely allow us to see the differences between the ``$\De$-bounds'' and the ``$H$-bounds'' more clearly, while taking less time and effort. 

We shall see that, even though the ``$H$-bounds'' $L_\FM(H)$ and $U_\FM(H)$ and the ``$\De$-bounds'' $L(\De)$ and $U(\De)$ are exact in terms of $H$ and $\De$, respectively, they have rather different properties. 

\begin{remark}\label{rem:comp}
Typically, the lower $H$-bound $L_\FM(H)$ on $p_*$ appears to be better (that is, larger) than the lower $\De$-bound $L(\De)$, whereas the upper $H$-bound $U_\FM(H)$ on $p_*$ appears to be worse (that is, larger) than the upper $\De$-bound $U(\De)$ and even its simplified but less accurate version $U_\simpl(\De)$. 

However, in some rather exceptional cases these relations are reversed. 

In particular, if the best possible misclassification probability $p_*$ is large enough, then the lower $\De$-bound $L(\De)$ may be better than the lower $H$-bound $L_\FM(H)$, for each $k\ge3$. 

On the other hand, if $k$ is large enough and $p_*$ is small enough, then the upper $\De$-bound $U(\De)$ may be worse than the upper $H$-bound $U_\FM(H)$. However, I have not been able to find cases with $U(\De)$ (or even $U_\simpl(\De)$) worse than $U_\FM(H)$ when there are at most $k=9$ classes. 
\end{remark}

More specifically, we have the following propositions. (As usual, $\ii{\cdot}$ will denote the indicator function.) 

\begin{proposition}\label{prop:comp-lo}
Suppose that $k\ge3$ and for each $x\in S$ the vector $(\rho_1(x),\dots,\rho_k(x))$ is a permutation of the vector $(a_1,\dots,a_k)$, where 
\begin{equation*}
a_i=\frac1\ell\,\ii{1\le i\le\ell} 	
\end{equation*}
for some natural $\ell\ge2$ in the set $\{k-3,k-2,k-1\}$ and for all $i=1,\dots,k$; one may also allow $\ell=k-4$ if $k\in\{6,7,8,9\}$. 
Then $L(\De)>L_\FM(H)$. 
\end{proposition}

\begin{proposition}\label{prop:comp-hi}
Fix any $\nu\in(1,\infty)$. 
Suppose that for each $x\in S$ the vector $(\rho_1(x),\dots,\rho_k(x))$ is a permutation of the vector $(a_1,\dots,a_k)$, where $k>\nu$ and 
\begin{equation*}
a_i=\Big(1-\frac{\nu-1}k\Big)\,\ii{i=1}+\frac{\nu-1}{k(k-1)}\,\ii{2\le i\le k} 	
\end{equation*}
for all $i=1,\dots,k$. 
Then $U(\De)>U_\FM(H)$ for all large enough $k$ (depending on the value of $\nu$). 
\end{proposition}  

Note that in Proposition~\ref{prop:comp-lo} the best possible misclassification probability $p_*=1-\frac1\ell$ is large, especially when $\ell$ is large (and hence so is $k$). In contrast, in  Proposition~\ref{prop:comp-hi} $p_*=\frac{\nu-1}k$ is small for the
large values of $k$, assumed in that proposition. 
Either of these two kinds of situations, especially the second one, may be considered somewhat atypical: it usually should be difficult to make the misclassification probability $p_*$ small when the number $k$ of possible classes is large; on the other hand, when $k$ is not very large, one may hope that the best possible misclassification probability is small enough.  

 
Concerning the case of two classes, we have 

\begin{proposition}\label{prop:k=2}
Suppose that $k=2$. Then $U(\De)=L(\De)=L_\FM(H)=p_*$ for all pairs of r.v.'s $(X,Y)$. So, one can say that the bounds $U(\De)$, $L(\De)$, and $L_\FM(H)$ always perfectly estimate the best possible misclassification probability $p_*$ -- if $k=2$. 

On the other hand, here $U_\FM(H)>U_\simpl(\De)>p_*$ unless there is a set $S_0\in\Si$ such that $\mu(S_0)=0$ and for each $x\in S\setminus S_0$ either $\rho_1(x)=\rho_2(x)=1/2$ or $\{\rho_1(x),\rho_2(x)\}=\{0,1\}$ -- that is, 
the values $\rho_1(x)$ and $\rho_2(x)$ constitute a permutation of the numbers $0$ and $1$. Thus, in the case $k=2$, with the mentioned trivial exceptions, even the simplified upper $\De$-bound $U_\simpl(\De)$ on $p_*$ is strictly better than the upper $H$-bound $U_\FM(H)$, but still $U_\simpl(\De)$ is not a perfect estimate of $p_*$. 
\end{proposition} 
 
An important case is that of three classes, so that $k=3$. Here, in the ``pure'' setting, for each $x\in S$ the triple $(\rho_1(x),\rho_2(x),\rho_2(x))$ is a permutation of the triple $(1-p,p-\vp,\vp)$, where $p:=p_*\in[0,1-1/3]$ and $1-p\ge p-\vp\ge\vp\ge0$ or, equivalently, $p\in[0,2/3]$ and $(2p-1)_+\le\vp\le p/2$, where $u_+:=\max(0,u)$. Each of the 6 pictures in Figure~\ref{fig:k=3} presents the graphs of the decimal logarithms of the bounds $L(\De)$, $U(\De)$, $U_\simpl(\De)$, $L_\FM(H)$, and $U_\FM(H)$ as functions of $\vp\in[(2p-1)_+,p/2]$ with the misclassification probability $p=p_*$ taking a fixed value in the set $\{0.01,0.1,0.3,0.5,0.6,0.64\}$. 
We see that in all these cases the upper $\De$-bound $U(\De)$ and even its simplified (but worse) version $U_\simpl(\De)$ are better than the upper $H$-bound $U_\FM(H)$, over the entire range of values of $\vp$. For small values of the best possible misclassification probability $p_*$, the lower $H$-bound $L_\FM(H)$ is significantly better than $L(\De)$ over all values of $\vp$; however, this comparison is reversed if $p_*$ is large enough but $\vp$ is small enough (especially in the case $p_*=0.5$).

An interesting series of cases is given by what may be called the binomial model (with a parameter $q\in(0,1)$), in which $k=2^m$ for a natural $m$, and for each $x\in S$ the vector $(\rho_1(x),\dots,\rho_k(x))$ is a permutation of a vector $(a_1,\dots,a_k)$, where each $a_i$ is of the form $(1-q)^jq^{m-j}$ for some $j\in\{0,\dots,m\}$, and the multiplicity of the form $(1-q)^jq^{m-j}$ among the $a_i$'s is $\binom mj$ for each $j\in\{0,\dots,m\}$. Clearly then, all the $a_i$'s are nonnegative, and $a_1+\dots+a_k=\sum_{j=0}^m\binom mj (1-q)^jq^{m-j}=1$. In particular, for $m=1$ we have $k=2$, and then we may take $(a_1,a_2)=(1-q,q)$. 
For $m=2$ we have $k=4$, and then we may take 
$$(a_1,a_2,a_3,a_4)=\big((1-q)^2,\,(1-q)q,\,(1-q)q,\,q^2\big).$$ 
Choosing, in the latter case, $S=\{1,2,3,4\}$ and $q=Q(\sqrt{2E_b/N_0})$, where $Q(x):=\int_x^\infty\frac1{\sqrt{2\pi}}\,e^{-u^2/2}du$ is the tail probability for the standard normal distribution, $E_b$ is the energy per bit, and $N_0/2$ is the noise power spectral density (PSD), we see that the resulting particular case of the binomial model covers the so-called 
quadrature phase-shift keying (QPSK) digital communication scheme over an additive white Gaussian noise (AWGN) channel (cf.\ e.g. \cite[page~313]{erdogmus-principe}), which in fact provided the motivation for the general binomial model. 

Another interesting series of cases is given by what may be called the exponential model (with a parameter $q\in(0,1)$), in which for each $x\in S$ the vector $(\rho_1(x),\dots,\rho_k(x))$ is a permutation of a vector $(a_1,\dots,a_k)$, where $a_i:=(1-q)^{i-1}q^{k-i}/c_q$ and $c_q:=c_{k,q}:=\sum_{i=1}^k(1-q)^{i-1}q^{k-i}$, so that all the $a_i$'s are nonnegative and $a_1+\dots+a_k=1$. Informally, the exponential model can be obtained from the binomial one by removing the multiplicities.  

Graphical comparisons of the ``$\De$-bounds'' with the ``$H$-bounds'' (as functions of the parameter $q$) for the cases $k=2,4,8$ of the binomial and exponential models are presented in Figure~\ref{fig:bin-exp}. Note here that, by symmetry, it is enough to consider $q\in(0,1/2]$. 
Obviously, for $k=2$ the binomial and exponential models are the same, and in this case they are the same as the essentially unique general ``pure'' model for $k=2$, fully considered in Proposition~\ref{prop:k=2}. Accordingly, the pictures in the first row in Figure~\ref{fig:bin-exp} are identical to each other, and the graphs of the bounds $U(\De)$, $L(\De)$, and $L_\FM(H)$ are the same as that of $p_*$. 
The cases $k=4,8$ in Figure~\ref{fig:bin-exp} illustrate the first sentence in Remark~\ref{rem:comp}. It appears that the comparisons in the exponential model are somewhat more favorable to the $\De$-bounds than they are in the binomial model. 


\begin{figure}[h]%
\includegraphics[width=\columnwidth]{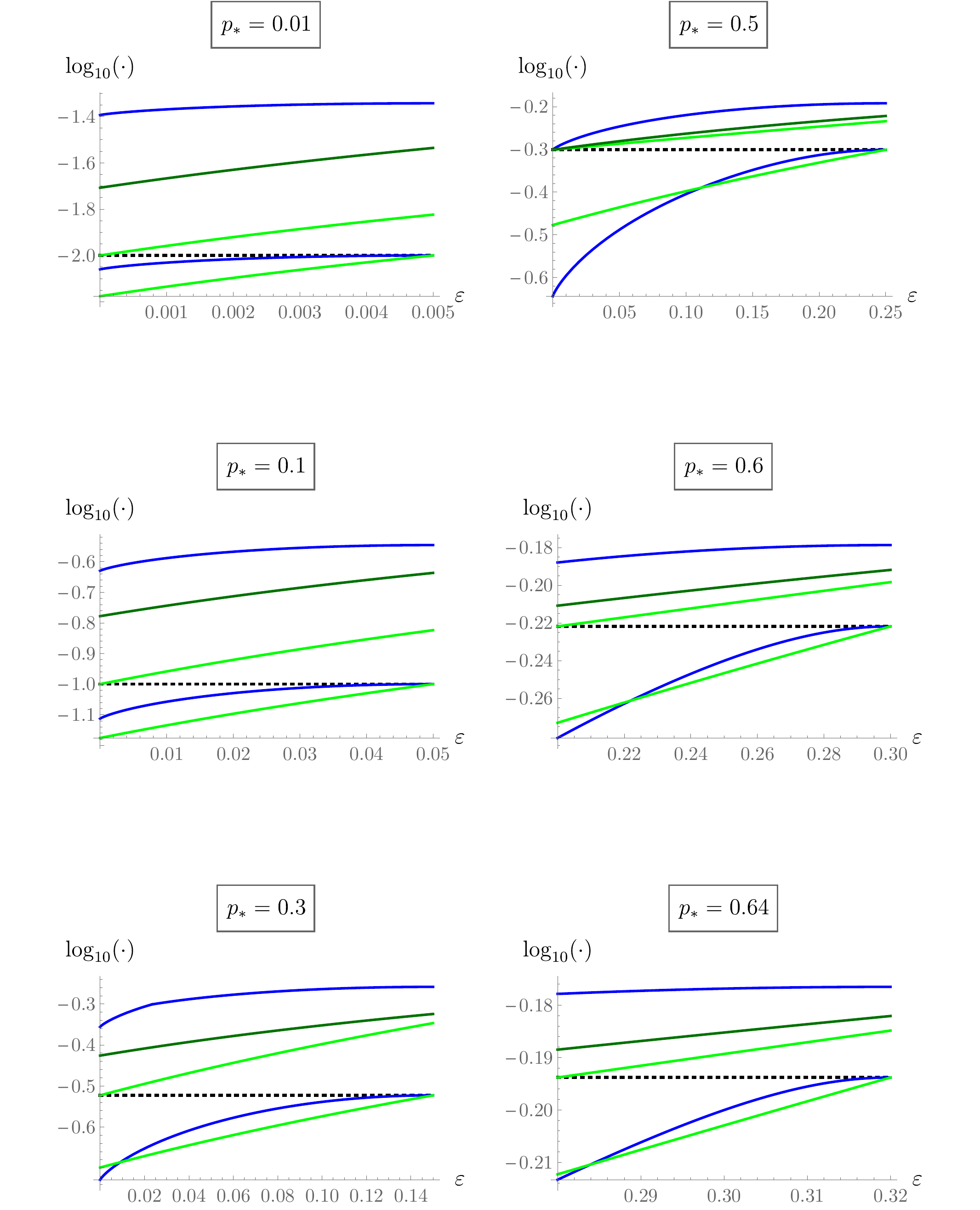}%
\caption{Graphs of $\log_{10}L(\De)$ (green), $\log_{10} U(\De)$ (green), $\log_{10} U_\simpl(\De)$  (dark green), $\log_{10} L_\FM(H)$  (blue), $\log_{10} U_\FM(H)$ (blue), and $\log_{10}p_*$ (dashed) for $k=3$ and $p_*\in\{0.01,0.1,0.3,\break 
0.5,0.6,0.64\}$.}%
\label{fig:k=3}%
\end{figure}

\newpage

\begin{figure}[h]%
\includegraphics[width=\columnwidth]{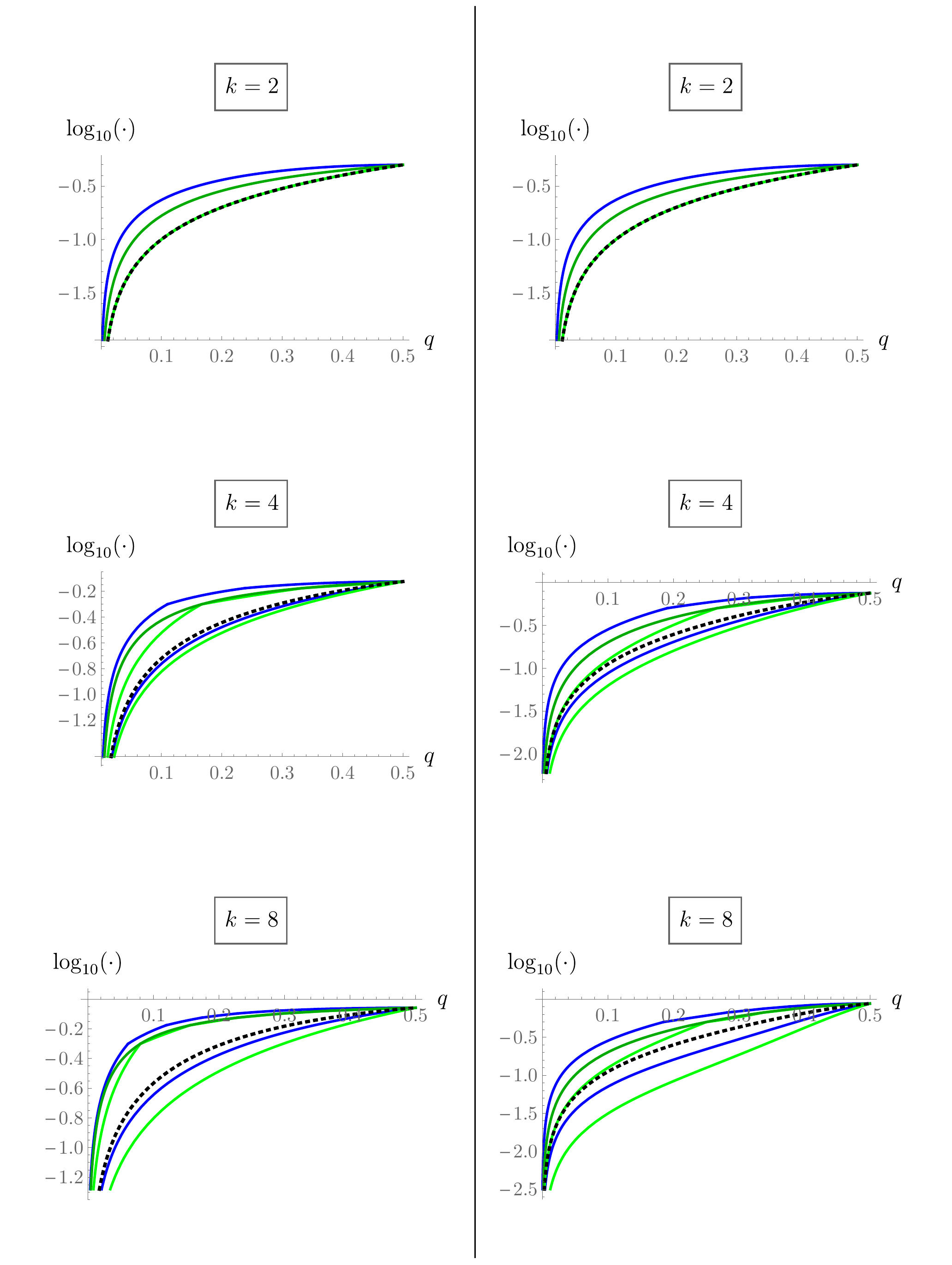}%
\caption{Graphs of $\log_{10}L(\De)$ (green), $\log_{10} U(\De)$ (green), $\log_{10} U_\simpl(\De)$  (dark green), $\log_{10} L_\FM(H)$  (blue), $\log_{10} U_\FM(H)$ (blue), and $\log_{10}p_*$ (dashed) for $k=2,4,8$.  
Left column: Binomial model. Right column: Exponential model.}  
\label{fig:bin-exp}%
\end{figure}

\hspace*{5cm}


\newpage

Upper and lower bounds on the best possible misclassification probability in terms of Renyi's conditional entropy were announced in \cite{erd-princ-proc}, where the input and output r.v.'s were denoted by $M$ and $W$, respectively, and both $M$ and $W$ were assumed to take values in the same set $\{1,\dots,k\}$.  Renyi's conditional entropy used in \cite{erd-princ-proc} was defined by the formula 
\begin{equation}
	H_\be(W|M):=\sum_{m=1}^k \P(M=m)H_\be(W|m),
\end{equation}	
where
\begin{equation}
	H_\be(W|m):=
	\frac1{1-\be}\,\log\sum_{w=1}^k \P(W=w|M=m)^\be,\quad\log:=\log_2, 
\end{equation}
and $\be\in(0,1)\cup(1,\infty)$, so that Shannon's conditional entropy  
\begin{equation}
	H_1(W|M):=H(W|M)=-\sum_{m=1}^k \P(W=w)\log\P(W=w|M=m)
\end{equation}
may be viewed as a limit case of Renyi's: $H_\be(W|M)\to H(W|M)$ as $\be\to1$. 
Note that $H_\be(W|M)$ is the conditional entropy of the output $W$ given the input $M$. 
Thus, for some reason, the standard roles of the input and output r.v.'s were reversed in \cite{erd-princ-proc}; cf. e.g.\ the conditional entropy 
\emph{of the input given the output} used in \cite[formula~(3)]{fed-merhav}. 

The mentioned upper and lower bounds announced in \cite{erd-princ-proc} were given by inequalities of the form 
\begin{equation}\label{eq:E-P}
\frac{H_\al(W|M)-H_S(e)}{N_1}\le P(e)\le\frac{H_\be(W|M)-H_S(e)}{N_2}, 	
\end{equation}
where $\be<1\le\al$, $N_1,N_2$ are some positive expressions, 
$$H_S(e):=-P(e)\log P(e)-(1-P(e))\log(1-P(e)),$$ 
and $P(e)$ is the ``the classification error probability'', which doers not seem to be explicitly defined in \cite{erd-princ-proc}. 
A proof of these bounds was offered later in \cite[Appendix]{erdogmus-principe}, from which it appears (see \cite[formula~(A.1)]{erdogmus-principe}) that $P(e)$ is understood as $\P(W\ne M)$. 
However, the proof in \cite{erdogmus-principe} of the upper bound on $P(e)$ in \eqref{eq:E-P} appears to be mistaken, and the upper bound itself may be negative and thus false in general. 

Indeed, the second inequality in \eqref{eq:E-P} appears to be obtained in \cite{erdogmus-principe} by multiplying the expressions in \cite[formula~(A.7)]{erdogmus-principe} by $p(m_k)$, then summing in $k$ (in the notations there), and finally using the inequality $\sum_k p(m_k)H_S(e|m_k)[=H_S(e|M)]\ge H_S(e)$. However, in general the reverse inequality is true: $H_S(e|M)\le H_S(e)$; cf.\ even the derivation of (A.6) from (A.5) in \cite{erdogmus-principe}. Also, the proof in \cite{erdogmus-principe} does not use the condition that $P(e)$ is \emph{the smallest possible} classification error probability, and without such a condition no reasonable upper bound on $P(e)$ is possible. 

More importantly, as mentioned above, the upper bound on $P(e)$ in \eqref{eq:E-P} is false in general. For a very simple counterexample, suppose that $k=2$, $\P(W=1,M=1)=\P(W=1,M=2)=1/2$, and  $\P(W=2,M=1)=\P(W=2,M=2)=0$. Then $P(e)=1/2$, $H_\be(W|M)=0$ for all $\be$, and $H_S(e)=1$, so that the presumed upper bound on $P(e)$ in \eqref{eq:E-P} is negative and thus false. 

Another two pairs of upper and lower bounds were announced in \cite{erd-princ-proc}, in terms of the joint Renyi's joint entropy $H_\be(W,M)$ of $(W,M)$ and Renyi's mutual information $I_\be(W,M)$ between $W$ and $M$, rather than Renyi's conditional entropy, with no apparent proofs for these additional bounds. However, the same simple example given above will quite similarly show that these additional upper bounds are false in general, too. 

As stated in the abstract in \cite{erdogmus-principe}, the mentioned bounds in \cite{erd-princ-proc} on $P(e)$ ``were practically incomputable'', because $P(e)$ itself appears in those bounds. Therefore, an effort was made in \cite{erdogmus-principe} to modify the bounds in \cite{erd-princ-proc} -- by making the upper bounds greater and the lower bounds smaller -- to make them computable. However, in view of what has been said, the modified upper bounds in \cite{erdogmus-principe} remain without a valid proof. As for the modified lower bounds, it is stated in \cite[page~313]{erdogmus-principe} that, in the examples considered there, ``the modified lower bounds
are not depicted because they turn out to be negative''. 

For all these reasons, we shall not attempt to compare our bounds with ones in \cite{erd-princ-proc,erdogmus-principe}.

\section{Proofs}\label{proofs}

\begin{proof}[Proof of Proposition~\ref{prop:f_*}] 
Clearly, $f_*$ is a map from $S$ to $[k]$. Also, $f_*$ is measurable, since $f_*^{-1}(\{y\})=B_y\setminus\bigcup_{z=1}^{y-1}B_z\in\Si$ for each $y\in[k]$, where $B_y:=\bigcap_{z=1}^k B_{y,z}$ and $B_{y,z}:=\{x\in S\colon\rho_y(x)\ge\rho_z(x)\}\in\Si$. 
Thus, $f_*$ is a classifier. 
Moreover, for any classifier $f$, 
\begin{align*}
	1-p_f=\P(f(X)=Y)
	&=\sum_{y=1}^k\P(Y=y,f(X)=y) \\ 
&=\sum_{y=1}^k\int_S\ii{f(x)=y}\mu_y(\d x) \\ 
	&=\sum_{y=1}^k\int_S\ii{f(x)=y}\rho_y(x)\,\mu(\d x) \\ 
	&	=\int_S\sum_{y=1}^k\ii{f(x)=y}\rho_y(x)\,\mu(\d x) \\ 
		&\le\int_S\max\limits_{y=1}^k\rho_y(x)\,\mu(\d x) =
		1-p_{f_*}.  
\end{align*}
This completes the proof of Proposition~\ref{prop:f_*}. 
\end{proof}

In view of Proposition~\ref{prop:f_*} and \eqref{eq:De}, 
Theorem~\ref{th:} and Proposition~\ref{prop:exact} follow immediately by the lemma below, with $\rho_i(x)$ in place of $a_i$.  

\begin{lemma}\label{lem:}
Suppose that 
\begin{equation}\label{eq:a_i}
\text{$a_1,\dots,a_k$ are nonnegative real numbers such that $\sum_1^k a_i=1$.}	
\end{equation}
Then 
\begin{equation}\label{eq:de-bounds}
	L(\de)\le1-\max\limits_1^k a_i\le U(\de)\le U_\simpl(\de), 
\end{equation}
where 
\begin{equation*}
	\de:=\sum_{1\le i<j\le k}|a_i-a_j|
\end{equation*}
and the functions $L$, $U$, and $U_\simpl$ are defined as in Theorem~\ref{th:}. 

Under the stated conditions on the $a_i$'s, one always has $0\le\de\le k-1$; cf.\ Remark~\ref{rem:De}. 

The bounds $L(\de)$ and $U(\de)$ on $1-\max\limits_1^k a_i$ are exact for each possible value of $\de$: 
\begin{enumerate}[(i)]
	\item For each $d\in[0,k-1]$, if 
\begin{equation}\label{eq:a_i-low}
	a_1=\frac{1+d}k\quad\text{and}\quad a_2=\dots=a_k=\frac1k-\frac d{k(k-1)}, 
\end{equation}
then condition \eqref{eq:a_i} holds, $\de=d$, \big[$\max\limits_1^k a_i=a_1$,\big] and the first inequality in \eqref{eq:de-bounds} turns into the equality. 
If the $a_i$'s satisfy condition \eqref{eq:a_i} but do not constitute a permutation of the $a_i$'s as in \eqref{eq:a_i-low} with $d=\de$, then the first inequality in \eqref{eq:de-bounds} is strict. 
	\item For each $d\in[0,k-1]$, if 
\begin{equation}\label{eq:a_i-hi}
	a_i=(1-U(d))\ii{i\le k-\ceil d}+\frac{\ceil d-d}{k+1-\ceil d}\ii{i=k+1-\ceil d} 
\end{equation}
for all $i\in[k]$, 
then condition \eqref{eq:a_i} holds, $\de=d$, \big[$\max\limits_1^k a_i=a_1$,\big] and the second inequality in \eqref{eq:de-bounds} turns into the equality. 
If the $a_i$'s satisfy condition \eqref{eq:a_i} but do not constitute a permutation of the $a_i$'s as in \eqref{eq:a_i-hi} with $d=\de$, then the second inequality in \eqref{eq:de-bounds} is strict.
\end{enumerate}  
\end{lemma}

\begin{proof}
It is quite easy to see that $U_\simpl(d)$ is concave in $d\in[0,k-1]$. Moreover, as noted in Remark~\ref{rem:bounds}, $U(d)$ is the linear interpolation of  $U_\simpl(d)$ over $d=0,\dots,k-1$. Thus, we have the last inequality in \eqref{eq:de-bounds}. 

It remains to establish the lower bound $L(\de)$ and upper bound $U(\de)$ on $1-\max\limits_1^k$ and to show that these bounds are attained, with $\de=d$, if and only if the $a_i$'s are as in \eqref{eq:a_i-low} and \eqref{eq:a_i-hi}, respectively. 

By symmetry, without loss of generality (w.l.o.g.) $a_1\ge\cdots\ge a_k$. Then, letting $h_i:=a_i-a_{i+1}$ for $i\in[k]$ (with $a_{k+1}:=0$), we have 
\begin{equation*}
	h_1\ge0,\dots,h_k\ge0,
\end{equation*}
\begin{equation*}
	\max\limits_1^k a_i=a_1=\sum_1^k h_i,
\end{equation*}
\begin{multline*}
	\de=\sum_{1\le i<j\le k}(a_i-a_j)=\sum_{1\le i<j\le k}\sum_{q=i}^{j-1}h_q
	=\sum_{q=1}^{k-1}h_q\sum_{1\le i\le q}\,\sum_{q+1\le j\le k}1 \\ 
	=\sum_{q=1}^{k-1}h_q\, q(k-q)=\sum_{i=1}^k i(k-i)h_i,  
\end{multline*}
\begin{equation*}
	1=\sum_1^k a_j=\sum_{j=1}^k\sum_{i=j}^k h_i=\sum_{i=1}^k ih_i. 
\end{equation*}

Take now indeed any $d\in[0,k-1]$. 
Introducing 
\begin{equation*}
	p_i:=ih_i
\end{equation*}
for $i\in[k]$, we further restate the conditions on the $a_i$'s (with $\de$ equal the prescribed value $d\in[0,k-1]$, as desired): 
\begin{equation}\label{eq:sum p_i}
	p_1\ge0,\dots,p_k\ge0,\, \sum_{i=1}^k p_i=1, 
\end{equation}
\begin{equation}\label{eq:sum ip_i}
	\sum_{i=1}^k (k-i)p_i=d\quad\text{or, equivalently,\quad }\sum_{i=1}^k ip_i=k-d, 
\end{equation}
and 
\begin{equation*}
	\max\limits_1^k a_i=a_1=\sum_1^k g(i)p_i, 
\end{equation*}
where $g(i):=\frac1i$; here and in the rest of the proof of Lemma~\ref{lem:}, $i$ is an arbitrary number in the set $[k]$. 

Introduce also  
\begin{equation*}
g^U(i):=g(k-m-1)+[g(k-m)-g(k-m-1)][i-(k-m-1)]	
\end{equation*} 
and 
\begin{equation*}
	p_i^U:=(d-m)\ii{i=k-m-1}+(m+1-d)\ii{i=k-m},\ 
\end{equation*}
where 
\begin{equation*}
	m:=\ceil d-1;
\end{equation*} 
here and in the rest of the proof of Lemma~\ref{lem:}, $i$ is an arbitrary number in the set $[k]$. One may note at this point that $m\in\{0,\dots,k-2\}$. 
Note that the function $g$ is strictly 
convex on the set $[k]$, the function $g^U$ is affine, $g^U=g$ on the set $\{k-m-1,k-m\}$, and hence $g>g^U$ on $[k]\setminus\{k-m-1,k-m\}$. Moreover, conditions \eqref{eq:sum p_i} and \eqref{eq:sum ip_i} hold with $p_i^U$ in place of $p_i$. So, 
\begin{multline}\label{eq:U chain}
\sum_1^k g(i)p_i\ge \sum_1^k g^U(i)p_i=g^U\Big(\sum_1^k ip_i\Big) \\ 
=g^U(k-d)
=g^U\Big(\sum_1^k ip_i^U\Big)=\sum_1^k g^U(i)p_i^U=\sum_1^k g(i)p_i^U; 		
\end{multline}
the inequality here holds because $g\ge g^U$; the first and fourth equalities follow because the function $g^U$ is affine; the second and third equalities hold because of the condition \eqref{eq:sum ip_i} for the $p_i$'s and $p_i^U$'s; and the last equality follows because $g^U(i)=g(i)$ for $i$ in the set $\{k-m-1,k-m\}$, whereas $p_i^U=0$ for $i$ not in this set. 
We conclude that, under conditions \eqref{eq:sum p_i} and \eqref{eq:sum ip_i}, $\max\limits_1^k a_i$ is minimized -- or, equivalently, $1-\max\limits_1^k a_i$ is maximized -- if and only if $p_i=p_i^U$ for all $i$; that is, if and only if  
$h_i=p_i^U/i$ or all $i$; that is, if and only if the $a_i$'s -- related to the $p_i$'s by the formula $a_i=\sum_{j=i}^k \frac1i\,p_i$ -- are as in \eqref{eq:a_i-hi}. 
This concludes the proof of the part of Lemma~\ref{lem:} concerning the upper bound $U(\cdot)$. 

The proof of the part of Lemma~\ref{lem:} concerning the lower bound $L(\cdot)$ is similar and even  easier. Here let 
\begin{equation*}
g^L(i):=g(1)+[g(k)-g(1)]\frac{i-1}{k-1}	
\end{equation*} 
and 
\begin{equation*}
	p_i^L:=\frac d{k-1}\,\ii{i=1}+\Big(1-\frac d{k-1}\Big)\ii{i=k}. 
\end{equation*} 
Recall that the function $g$ is strictly 
convex on the set $[k]$. Note that the function $g^L$ is affine, $g^L=g$ on the set $\{1,k\}$, and $g^L>g$ on the set $[k]\setminus\{1,k\}$, so that  $g\le g^L$ on $[k]$. Moreover, conditions \eqref{eq:sum p_i} and \eqref{eq:sum ip_i} hold with $p_i^L$ in place of $p_i$. So, 
\begin{multline*}
\sum_1^k g(i)p_i\le \sum_1^k g^L(i)p_i=g^L\Big(\sum_1^k ip_i\Big) \\ 
=g^L(k-d)
=g^L\Big(\sum_1^k ip_i^L\Big)=\sum_1^k g^L(i)p_i^L=\sum_1^k g(i)p_i^L;  		
\end{multline*}
cf.\ \eqref{eq:U chain}. 
So, under conditions \eqref{eq:sum p_i} and \eqref{eq:sum ip_i}, $\max\limits_1^k a_i$ is maximized -- or, equivalently, $1-\max\limits_1^k a_i$ is minimized -- if and only if $p_i=p_i^L$ for all $i$; that is, if and only if 
$h_i=p_i^L/i$ or all $i$; that is, if and only if the $a_i$'s are as in \eqref{eq:a_i-low}. 
This concludes the proof of the part of Lemma~\ref{lem:} concerning the upper bound $L(\cdot)$ and hence the entire proof of the lemma. 
\end{proof}

\begin{proof}[Proof of Proposition~\ref{prop:comp-lo}]
We have to show that, under the conditions of this proposition, $L(\De)>L_\FM(H)$. Recalling \eqref{eq:L_FN} and the fact that the function $\Phi$ is increasing, we can rewrite inequality $L(\De)>L_\FM(H)$ as $\Phi(L(\De))>H$. In view of \eqref{eq:H}, \eqref{eq:De}, and \eqref{eq:L,U}, $H=\ln\ell$, $\De=k-\ell$, and $L(\De)=\frac{\ell-1}k$. 
So, inequality $\Phi(L(\De))>H$ can be in turn rewritten as 
\begin{equation}\label{eq:comp-lo}
	d(\ell):=d_k(\ell):=\Phi\Big(\frac{\ell-1}k\Big)-\ln\ell\overset{\text{(?)}}>0
\end{equation}
for $k$ and $\ell$ as in the conditions in Proposition~\ref{prop:comp-lo}. 
For $k\in\{6,7,8,9\}$ and $\ell=k-4$, as well as for $k\in\{3,4,5\}$ and $\ell\ge2$ in the set $\{k-3,k-2,k-1\}$,
inequality \eqref{eq:comp-lo} can be verified by direct calculations. So, without loss of generality $k\ge6$ and $\ell\in[k-3,k)$, where we may allow $\ell$ to take non-integral values as well. Note that 
\begin{equation*}
\begin{aligned}
	d''(\ell)(\ell-1) \ell^2 (k+1 - \ell)
	&=-1 + 2 \ell - 2 \ell^2 + (\ell-1) k \\ 
	&\le
	-1 + 2 \ell - 2 \ell^2 + (\ell-1) (\ell+3)=-(\ell-2)^2<0
\end{aligned}
\end{equation*}
for $\ell>2$. Hence, $d(\ell)$ is strictly concave in $\ell>2$ such that $\ell\in[k-3,k]$. Also, $d(k)=0$. 
So, to complete the proof of Proposition~\ref{prop:comp-lo}, it suffices to show that 
\begin{equation}\label{eq:td}
	\tilde d(k):=k\,d_k(k-3)
	\big[=(k-4) \ln \tfrac{k(k-1)}{k-4}+4 \ln\tfrac{k}{4}-k \ln (k-3)\big]
   \overset{\text{(?)}}>0
\end{equation}
for $k\ge6$. We find that 
\begin{equation*}
	\tilde d''(k)=-\frac{36}{(k-4) (k-3)^2 (k-1)^2 k}<0
\end{equation*}
for $k\ge6$, and so, $\tilde d(k)$ is strictly concave in $k\ge6$. Moreover, $\tilde d(6)=0.446\dots>0$ and $\tilde d(k)\to6 - 8 \ln2=0.454\dots>0$ as $k\to\infty$. 
Thus, inequality \eqref{eq:td} indeed holds for $k\ge6$, and the proof of Proposition~\ref{prop:comp-lo} is now complete. 
\end{proof} 

\begin{proof}[Proof of Proposition~\ref{prop:comp-hi}]
Under the conditions of this proposition, 
\begin{equation*}
	H=-\Big(1-\frac{\nu-1}k\Big)\ln\Big(1-\frac{\nu-1}k\Big)-\frac{\nu-1}k\,\ln\frac{\nu-1}{k(k-1)}\to0
\end{equation*}
and hence, by \eqref{eq:U_FN} and \eqref{eq:e(H)}, $U_\FM(H)\to0$ 
as $k\to\infty$. 

On the other hand, here $\De=k-\nu$. Therefore and because $U(\De)$ is decreasing in $\De$,  
\begin{equation*}
	U(\De)\ge U(\ceil\De)=U_\simpl(\ceil\De)=1-\frac1{k-\ceil\De}=1-\frac1{\lfloor\nu\rfloor},
\end{equation*}
which latter is a positive constant with respect to $k$ and hence does not go to $0$ as $k\to\infty$. Thus, the conclusion of Proposition~\ref{prop:comp-hi} follows. 
\end{proof}

\begin{proof}[Proof of Proposition~\ref{prop:k=2}] 
Here w.l.o.g.\ $\{\rho_1(x),\rho_2(x)\}=\{1-p,p\}$ for each $x\in S$, where $p:=p_*\in[0,1/2]$. Then the equalities $U(\De)=L(\De)=L_\FM(H)=p_*$ follow immediately from the definitions. 

It remains to show that $U_\FM(H)>U_\simpl(\De)>p$ for $p\in(0,1/2)$. The second inequality here is obvious, since in this case $U_\simpl(\De)=\frac{2p}{1+2p}$. To verify that $U_\FM(H)>U_\simpl(\De)$ for $p\in(0,1/2)$, consider $d(p):=U_\FM(H)-U_\simpl(\De)
=-\frac12\,(1-p)\log_2(1-p)-\frac12\,p\log_2 p-\frac{2p}{1+2p}$. 
It is easy to see that 
$$d''(p)(1 - p) p (1 + 2 p)^3 \ln4
= -1 + p (16 \ln2-6)-p^2 (12 + 16 \ln2)- 8 p^3<0$$ 
for $p\in(0,1/2)$, so that $d$ is strictly concave on $(0,1/2)$. Also, $d(0+)=d(1/2)=0$. So, $d>0$ on $(0,1/2)$, which completes the proof of Proposition~\ref{prop:k=2}.     
\end{proof}

In conclusion, let us mention a sample of other related results found in the literature. 
In \cite{toussaint}, for $k=2$, sharp lower bounds on the misclassification probabilities
for three particular classifiers in terms of characteristics generalizing the Kullback--Leibler divergence and the Hellinger distance we obtained.
Lower and upper bounds on the misclassification probability based on Renyi's information were given in \cite{erdogmus-principe}. 
Upper and lower bounds on the 
risk of an empirical risk minimizer for $k=2$ were obtained in \cite{massart-nedelec} and \cite{pin-kontor}, respectively. 





\bibliographystyle{abbrv}
\bibliography{P:/pCloudSync/mtu_pCloud_02-02-17/bib_files/citations12.13.12}


\end{document}